\documentclass[11pt]{amsart}
\usepackage{amsmath,amssymb,amsfonts,amsthm,amscd}
\usepackage{mathrsfs,latexsym,url,graphicx,setspace,color,cancel}
\newtheorem{theorem}{Theorem}
\newtheorem*{theorem*}{Theorem}
\newtheorem{proposition}{Proposition}
\newtheorem{lemma}{Lemma}

\theoremstyle{definition}
\newtheorem{remark}{Remark}

\newtheorem{example}{Example}

\newcommand{\abs}[1]{\left\lvert#1\right\rvert}
\newcommand{\norm}[1]{\lvert\lvert#1\rvert\rvert}
\newcommand{\CP}{\mathcal{P}}
\newcommand{\R}{\mathbb{R}}

\newcommand{\B}{\mathcal{B}}
\newcommand{\disc}{\mathbb{D}}
\newcommand{\C}{\mathbb{C}}
\newcommand{\n}{\mathbb{N}}

\newcommand{\dD}{\mathcal{D}}

\newcommand{\zb}{\overline{z}}
\newcommand{\D}{\Omega}
\newcommand{\Dc}{\overline{\Omega}}
\newcommand{\dbar}{\overline{\partial}}

\onehalfspace

\title[Analysis on the Intersection of Pseudoconvex Domains]{Analysis on the Intersection of Pseudoconvex Domains}

\author{Mehmet \c{C}el\.ik}
\address[Mehmet \c{C}elik]{Texas A\&M University - Commerce, Department of Mathematics, 1600 Education Dr., Binnion Hall Room 303A, Commerce, TX 75429-3011, U.S.A.}
\email{mehmet.celik@tamuc.edu}

\author{Yunus E. Zeytuncu}
\address[Yunus E. Zeytuncu]{University of Michigan - Dearborn, Department of Mathematics and Statistics, Dearborn, MI  48128}
\email{zeytuncu@umich.edu}


\thanks{The work of the second author was partially supported by a grant from the Simons Foundation (\#353525), and also by a University of Michigan-Dearborn CASL Faculty Summer Research Grant. The work of the first author was partially supported by a Texas A$\&$M University - Commerce International Faculty Development Grant}

\begin{document}
\begin{abstract}
In this note, we discuss the preservation of certain analytic properties of the $\dbar$-Neumann operator, Bergman projection and Hankel operators on the intersection of pseudoconvex domains. 
\end{abstract}

\maketitle
\section{Introduction}
It has been well known that any obstruction for compactness of the $\dbar$-Neumann problem should live in the boundary of the domain of definition \cite[Section 4.8]{StraubeBook}. A test question to make this more precise is the following. Suppose we have two pseudoconvex domains where the respective $\dbar$-Neumann operators on both domains are compact. If the obstruction lives in the boundary, then the obstruction should be absent on the boundary of each domain and also the boundary of the intersection domain. Therefore, can we conclude that the $\dbar$-Neumann operator on the intersection domain is also compact? 

One challenge with this problem is that the intersection domain does not have smooth boundary. It is not known whether smooth forms are dense in  $\mbox{Dom}(\overline{\partial})\cap\mbox{Dom}(\overline{\partial}^{\ast})$ under the graph norm on a pseudoconvex domain with Lipschitz boundary. However, even getting a compactness estimate for smooth forms is not immediate, see \cite{StraubeAyyuru} for a recent partial answer. 

Inspired by this problem, similar questions about preservation of other analytic properties of different operators can be investigated on the intersection of two pseudoconvex domains. In this paper, we present some results related to this general investigation. In particular, we discuss the following directions.
\begin{itemize}
\item Compactness of the $\dbar$-Neumann operator on intersection domains.
\item Sobolev and $L^p$ regularity of the Bergman projection on intersection domains.
\item Hilbert-Schmidt properties of Hankel operators on intersection domains.
\end{itemize}

In the second section, we present some observations on the compactness problem on intersection domains. In the third section, we show by an elementary example that neither Sobolev nor $L^p$ regularity of the Bergman projection is necessarily preserved on the intersection of two domains. In the last section, we present two domains that both admit Hilbert-Schmidt Hankel operators but the intersection domain does not.


\section{Compactness on the intersection of two domains}
Let $\D$ be a bounded pseudoconvex domain in $\C^n$ with smooth boundary. \textit{A Compactness estimate} for the $\overline{\partial}$-Neumann operator is said to hold on $\Omega$ if for a given  $\varepsilon>0$ there is a constant $C_{\varepsilon}>0$ such that \begin{eqnarray*}\label{eq1}
\norm{u}^{2}\leq \varepsilon\left(\norm{\overline{\partial}u}^{2}+\norm{\overline{\partial}^{\ast}u}^{2}\right)+C_{\varepsilon}\norm{u}_{-1}^{2}
\end{eqnarray*}
is valid $\forall u\in \mbox{Dom}(\overline{\partial})\cap\mbox{Dom}(\overline{\partial}^{*})\subset L_{(0,q)}^{2}(\Omega)$. ($\norm{\cdot}_{-1}$ is the $L^{2}$-Sobolev ($-1$)-norm.)
Let $\Omega_{1}:=\left\{z\in \C^{n}\ |\ \rho_{1}(z)<0\right\}$ and $\Omega_{2}:=\left\{z\in \C^{n}\ |\ \rho_{2}(z)<0\right\}$ be two bounded pseudoconvex domains in $\C^{n}$ with smooth boundaries, and $\nabla\rho_{1}$ and $\nabla\rho_{2}$ be nonzero on $b\Omega_{1}$ and $b\Omega_{2}$ respectively. Assume that the compactness estimates for the $\overline{\partial}$-Neumann operator exist on both domains, $\Omega_{1}$ and $\Omega_{2}$. We investigate if there is a compactness estimate for the $\overline{\partial}$-Neumann operator on the intersection of $\Omega_{1}$ and $\Omega_{2}$. The local property of the compactness of the $\overline{\partial}$-Neumann operator (see \cite{FuStraube2001}) implies that local compactness estimates hold away from the set $S:=\lbrace z\in\C^{n}\ |\ \rho_{1}(z)=0=\rho_{2}(z)\rbrace.$ In the following two subsections, we present partial answers under additional assumptions. 

\subsection{Transversal Intersection}\label{SpecialDomains}

First, we assume that two domains $\Omega_1$ and $\Omega_2$ intersect transversally. That is, $S$ is a smooth manifold. We denote this by $\Omega_1\pitchfork \Omega_2$.

\begin{remark} 
Following the proof of locality of compactness estimate in \cite[Proposition 4.4]{StraubeBook}, one can see that the $\overline{\partial}$-Neumann operator is compact on $\Omega_{1}\pitchfork\Omega_{2}$ if one of the domains additionally satisfies property ($P$). In particular, if one of the domains is locally convexifiable domain or Hartogs in $\mathbb{C}^2$ then a compactness estimate holds on the intersection: property (P) is known to actually be equivalent to compactness on such domains \cite{FuStraube2001,FuChrist05}. The same conclusion holds if we consider property $(\tilde{P})$ instead of property $(P)$, see  \cite[Theorem 4.1.2]{AyyuruPhD2014} .
\end{remark}

Next, we focus on $\mathbb{C}^2$. In this case, the set $S$ is a two real dimensional smooth submanifold. If a point $p\in S$ has a non-trivial complex tangent space $H_{p}(S)$ we call it an exceptional point of $S$. We recall the following result about totally real manifolds in $\mathbb{C}^2$.

\begin{lemma}\cite[Lemma $17.2$]{HerbertWermer98}\label{lemma1}
Let $S$ be a totally real smooth submanifold of an open set in $\C^{2}$. Let $d_{S}(x):=dist(x,S)=\inf\lbrace \vert x-y\vert\ \vert\ y\in S \rbrace$. Then, there is a neighborhood $U_{_{S}}$ of $S$ such that $d_{S}^{2}(x)$ is smooth and strictly plurisubharmonic in $U_{_{S}}$.
\end{lemma}

\begin{example}\label{example7} Let $K$ be the set of all exceptional points of $S$. Note that $K$ is a compact subset of $S$. Suppose that $K\neq\emptyset$ and $S\backslash K\not=\emptyset$. Also suppose that $K$ has property $(P)$. Then we get a compactness estimate on the intersection. Indeed, $S\backslash K$ is a smooth manifold with real dimension $2$. By Lemma \ref{lemma1} we can say that on a neighborhood $U_{L}$ of any compact $L\subset S\backslash K$, there is a smooth strictly plurisubharmonic function,  $d_{L}^{2}(x)$. Then it is easy to see that $L$ satisfies property $(P)$. Namely, we have $\left(\frac{\partial^{2} d_{L}^{2}}{\partial z_{j}\partial \overline{z}_{k}}(z)\right)_{j,k}\geq C>0$ for $z\in U_{L}$. Thus, for a given $M>0$ set $\lambda_{M}(z):=\frac{2M}{C}d_{L}^{2}(z)$ on $z\in U_{L}$, then  $\left(\frac{\partial^{2}\lambda_{M}}{\partial z_{j}\partial \overline{z}_{k}}(z)\right)_{j,k}\geq M$ on $U_{L}$. There is a neighborhood $U_{_{M}}(\subset U_{L})$ of $L$ with $0\leq \lambda_{M}(z)\leq 1$. It is possible to write $S\backslash K$ as a union of countably many compact subsets $\left\{L_{j}\right\}_{j=1}^{\infty}$, where each $\overline{L}_{j}=L_{j}\subset\subset S\backslash K$ and has property $(P)$. We write $S\left(=\overline{S}\subset\subset\C^{2}\right)$ as $\left(\bigcup_{j=1}^{\infty} L_{j}\right)\cup K$. Since, each of these compact subsets has property $(P)$ then, by \cite[Proposition 1.9]{Sibony87} $S$ has property $(P)$ and hence a compactness estimate holds on the intersection. 

We know a few instances where $K$ will have the desired property. In particular, if 
\begin{itemize}
\item[$(a)$] $K$ is a discrete set,
\item[$(b)$] $K$ is a smooth curve,
\item[$(c)$] $K$ has $2$-dimensional Hausdorff measure zero,
\end{itemize}
then $K$ has property ($P$). Indeed, $(c)$ implies $(a)$ and $(b)$; for $(c)$ we refer to \cite{Sibony87} and \cite{Boas88}. 
\end{example}

For $p\in S$, let $T_{p}(S)$ denote the real tangent space at the point $p\in S$. We have $\mbox{dim}_{_{\R}}(T_{p}(b\Omega_{j}))=3$ for $j=1,\ 2$ and  $\mbox{dim}_{_{\R}}(H_{p}(b\Omega_{j}))=2$ for $j=1,\ 2$. Then 
$$\mbox{dim}_{_{\R}}(T_{p}(S))=\mbox{dim}_{_{\R}}\left(T_{p}(b\Omega_{1})\pitchfork T_{p}(b\Omega_{2})\right)=2.$$ 
Thus, if complex tangents exist at a point $p$ on $S$ then $T_{p}(S)=H_{p}(S)$. In other words, if the complex normals are linearly dependent (over $\C$), then and only then, we have a complex tangent to $S$ at $p$. Therefore, we conclude the following statement.
\begin{lemma}\label{totally real}
$p\in S$ is not an exceptional point if and only if 
\begin{eqnarray*}
\partial\rho_{1}(p)\wedge\partial\rho_{2}(p)\neq 0.
\end{eqnarray*}
$($That is, $S$ is totally real at $p$ if and only if $\mbox{det}\left(\frac{\partial\rho_{j}}{\partial z_{k}}(p)\right)_{1\leq j,k\leq 2}\neq 0$.$)$
\end{lemma}

\begin{remark}
Assume the set of exceptional points $K$ has an inner point (relative to the set $S$), that is,  $\mbox{Interior}(K)=:K^{\circ}\not=\emptyset$. Now, $K^{\circ}$ as a subset in $\C^{2}$ is a real smooth submanifold of $S$ all of whose tangents are complex tangents. Such a submanifold is a Riemann surface see \cite{BaouendiEbenfeltRothschild99}.
Thus, we would have an analytic disc on the boundaries of $\Omega_{1}$ and $\Omega_{2}$. In $\mathbb{C}^2$, existence of an analytic disc in the boundary contradicts the compactness of the $\dbar$-Neumann operators on $\D_1$ and $\D_2$ in the assumption. Therefore the set of exceptional points should have empty interior according to the relative topology on $S$.
\end{remark}

\begin{remark}
A resent result of Ayy\"ur\"u and Straube \cite{StraubeAyyuru} says if we have two smooth bounded pseudoconvex domains $\D_1$ and $\D_2$ (in $\C^n$) intersecting transversaly (so that the intersection is connected) and the $\dbar$-Neumann operators on $(0,n-1)$-forms both on $\D_1$ and on $\D_2$ are compact then so is the $\dbar$-Neumann operator on $(0,n-1)$-forms on the intersection. Note that when $n=2$, this result is sufficient to answer the main question in affirmative. That is, if $\D_1\pitchfork\D_2\subset\C^2$ and $N_{1}^{\D_1}$, $N_{1}^{\D_2}$ are compact then $N_{1}^{\D_1\pitchfork\D_2}$ is compact. \end{remark}

\medskip

\subsection{Non-transversal Intersection}
For $u\in$Dom$(\dbar^*)$ on $\D_1\cap\D_2$, rewriting the form $u$ as a sum of two forms $u_1$ and $u_2$ where $u_1\in$Dom$(\dbar^*)$ on $\D_1$ and $u_2\in$Dom$(\dbar^*)$ on $\D_2$ is a crucial decomposition in solving the intersection problem. If one can accomplish this kind of a decomposition, then a compactness estimate can be deduced on the intersection domain. However, a naive decomposition by smooth cutoff functions does not preserve Dom$(\dbar^*)$ on transversal intersections. On the other hand, in some cases of non-transversal intersection we can accomplish this decomposition. In this section, we present two instances where this happens. Here, we assume that two domains intersect non-transversally, but with smooth separation of boundaries. 

First, we consider the following special case. Let $\Omega_{1}$ and $\Omega_{2}$ intersect each other such that the boundary of $S$, $bS=S_1\cup S_2$, is the union of two disjoint boundary components, $S_{1}$ and $S_{2}$ such that $S_{1}\cap S_{2}=\emptyset$. Assume that the boundary of the resultant domain, $b(\Omega_{1}\cap\Omega_{2})$, is a piecewise smooth boundary and the non-smooth parts of $\Omega_{1}\cap\Omega_{2}$ are $S_{1}$ and $S_{2}$, as in the Figure \ref{Fig2}. In this setting, we prove the following proposition.

\begin{figure}
\centering
\includegraphics[width=130mm]{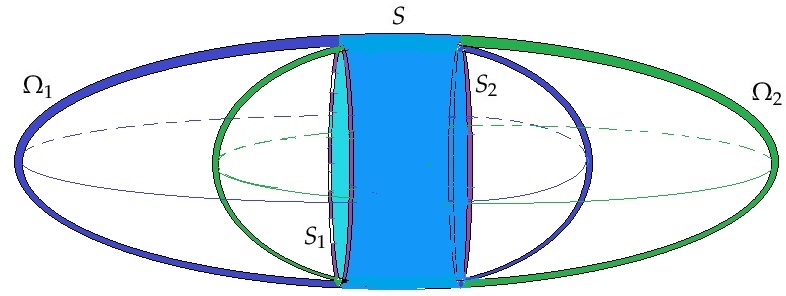}
\caption{$\Omega_{1}$ and $\Omega_{2}$ intersect each other in a way that the boundary of $S:=\lbrace\rho_1(z)=\rho_{2}(z)=0\rbrace\subset b(\D_1\cap\D_2)$ is union of two disjoint boundary components, $S_{1}$ and $S_{2}$.}
\label{Fig2}
\end{figure}



\begin{proposition}\label{nontransversal intersection}
If there exists a compactness estimate for the $\overline{\partial}$-Neumann operator on $\Omega_{1}$ and on $\Omega_{2}$, then there exists a compactness estimate for the $\overline{\partial}$-Neumann operator on $\Omega_{1}\cap\Omega_{2}$.
\end{proposition}

\begin{proof}
$S:=\lbrace\rho_1(z)=\rho_{2}(z)=0\rbrace\subset b(\D_1\cap\D_2)$ is a part of the boundary with co-dimension 1 and boundary of $S$ is $S_1\cup S_2$ with codimension $2$.

Let $K_{1}:=S_1\cup\left(\lbrace \rho_2(z)=0\rbrace\cap\Omega_{1}\right)$ and $K_{2}:=S_2\cup\left(\lbrace \rho_1(z)=0\rbrace\cap\Omega_{2}\right)$. The sets $K_{1}$ and $K_{2}$ are relatively disjointly closed in $\overline{\D_1\cap\D_2}$. Then, we can find a smooth function $\phi(z)\in C^{\infty}(\C^n)$ such that $0\leq \phi\leq 1$ with $\phi\equiv 1$ on a neighborhood of $K_{1}$ and $\phi\equiv 0$ on a neighborhood of $K_{2}$.

For $u\in \mbox{Dom}(\overline{\partial})\cap\mbox{Dom}(\overline{\partial}^{\ast})\subset L_{(0,1)}^{2}(\Omega_{1}\cap \Omega_{2})$ write $u=\phi u-(\phi-1)u$, and let $v_{1}:=(\phi-1)u$ and $v_{2}:=\phi u$.

Since $u\in \mbox{Dom}(\overline{\partial}^{\ast})\subset L_{(0,1)}^{2}(\Omega_{1}\cap \Omega_{2})$,  (due to the density lemma \cite[Lemma 4.3.2]{ChenShawBook} we can work with continuous up to the boundary forms $u$) the normal component (denoted by $u_n$) is zero on $b(\D_1\cap \D_2)$. 

Since $v_1\equiv 0$ on $K_1$, we extend $v_1$ as a zero form to the part $\Dc_1\backslash\Dc_2$ of the domain $\D_1$. Thus, $v_1\equiv 0$ on $\left(b\D_1\backslash \Dc_2\right)\cup S_1\subset b\D_1$ and the normal component of $v_1$ is $(\phi-1)u_n=(\phi-1)\cdot 0=0$ on $S\cup K_2\subset b\D_1$; note that $\left\{\left(b\D_1\backslash \Dc_2\right) \cup S_1\right\} \cup \left\{S\cup K_2\right\}=b\D_1$. Thus, $v_1\in\mbox{Dom}(\overline{\partial}^{\ast})\subset L_{(0,1)}^{2}(\Omega_{1})$. 

Since $u\in\mbox{Dom}(\overline{\partial})\subset L_{(0,1)}^{2}(\D_{1}\cap \D_{2})$ and $(\phi-1)$ is a smooth function with support on $\overline{\D_1\cap\D_2}$, $v_{1}:=(\phi-1)u\in\mbox{Dom}(\overline{\partial})\subset L_{(0,1)}^{2}(\D_{1}\cap \D_{2})$. By definition $v_1\equiv 0$ on $K_1$, we extend $v_1$ as a zero form to the part $\Dc_1\backslash\Dc_2$ of the domain $\D_1$. Then, $v_1\in L_{(0,1)}^{2}(\D_1)$. Moreover, $\dbar v_1=(1-\phi)\dbar u-\dbar \phi\wedge u \in L_{(0,2)}^{2}(\D_1)$ since $(1-\phi)$ and $\dbar \phi$ are smooth functions with support on $\overline{\D_1\cap\D_2}$ and zero on $K_1$. 

Thus, $v_{1}:=(\phi-1)u\in \mbox{Dom}(\overline{\partial})\cap\mbox{Dom}(\overline{\partial}^{\ast})\subset L_{(0,1)}^{2}(\Omega_{1})$. Similarly, one can show that $v_{2}:=\phi u \in \mbox{Dom}(\overline{\partial})\cap\mbox{Dom}(\overline{\partial}^{\ast})\subset L_{(0,1)}^{2}(\Omega_{2})$

From the hypothesis $\forall\varepsilon>0$ $\exists C_{\varepsilon}>0$ such that for $j=1,2$
\begin{eqnarray}
\norm{v_{j}}_{\Omega_{j}}^{2}\leq \varepsilon(\norm{\overline{\partial}v_{j}}_{\Omega_{j}}^{2}+\norm{\overline{\partial}^{\ast}v_{j}}_{\Omega_{j}}^{2})+C_{\varepsilon}\norm{v_{j}}_{-1,\Omega_{j}}^{2}.
\end{eqnarray}
Thus, $\forall\varepsilon>0$ $\exists C_{\varepsilon}>0$ such that for $j=1,2$
\begin{eqnarray}\label{eq53}
\norm{v_{j}}_{\Omega_{1}\cap\Omega_{2}}^{2}\leq \varepsilon(\norm{\overline{\partial}u}_{\Omega_{1}\cap\Omega_{2}}^{2}+\norm{\overline{\partial}^{\ast}u}_{\Omega_{1}\cap\Omega_{2}}^{2}+\norm{(\nabla\phi) u}_{\Omega_{1}\cap\Omega_{2}}^{2})+C_{\varepsilon}\norm{u}_{-1,\Omega_{1}\cap\Omega_{2}}^{2}.
\end{eqnarray} 

Next we compare the $(-1)$-Sobolev norms on $\Omega_j$'s and $\Omega_1\cap\Omega_2$. In particular,
\begin{eqnarray*}
\norm{v_{2}}_{-1,\Omega_{2}}^{2}&=&\sup\ \left\{\abs{\left(v_{2},\psi \right)_{\Omega_{2}}}\ :\ 0\neq\psi\in\left( W_{(0,q)}^{1}(\Omega_{2})\right)_{0}\ \text{ and }\ \norm{\psi}_{1,\Omega_{2}}=1\right\}\\
\nonumber &=&\sup\ \left\{\abs{\left(u,\phi\psi \right)_{\Omega_{2}}}\ :\ 0\neq\psi\in\left( W_{(0,q)}^{1}(\Omega_{2})\right)_{0}\ \text{ and }\ \norm{\psi}_{1,\Omega_{2}}=1\right\}\\
\nonumber &\leq& \norm{\phi\psi}_{1,\Omega_{1}\cap \Omega_{2}}\norm{u}_{-1,\Omega_{1}\cap \Omega_{2}}\\
\nonumber &\leq& C_{\phi}\norm{\psi}_{1,\Omega_{2}}\norm{u}_{-1,\Omega_{1}\cap \Omega_{2}}.
\end{eqnarray*}
By the same way we can get 
\begin{eqnarray}
\norm{v_{1}}_{-1,\Omega_{1}}^{2}\leq C_{(1-\phi)}\norm{\psi}_{1,\Omega_{1}}\norm{u}_{-1,\Omega_{1}\cap \Omega_{2}}.
\end{eqnarray}
Now, consider the basic estimate on $\Omega_{1}\cap\Omega_{2}$, 
\begin{eqnarray}\label{best2}
\norm{u}_{\Omega_{1}\cap \Omega_{2}}^{2}&\leq& C(\norm{\overline{\partial} u}_{\Omega_{1}\cap \Omega_{2}}^{2}+\norm{\overline{\partial}^{\ast}u}_{\Omega_{1}\cap \Omega_{2}}^{2})\\ 
\nonumber \text{for all}\ u &\in&\mbox{Dom}(\overline{\partial})\cap\mbox{Dom}(\overline{\partial}^{\ast})\subset L_{(0,q)}^{2}(\Omega_{1}\cap \Omega_{2}).
\end{eqnarray}
By using (\ref{best2}), we can estimate $\norm{\nabla\phi u}_{\Omega_{1}\cap\Omega_{2}}^{2}$ in (\ref{eq53}):
\begin{eqnarray}\label{eq56}
\nonumber\norm{(\nabla\phi) u}_{\Omega_{1}\cap\Omega_{2}}^{2}&\leq& \max_{z\in\overline{\Omega_{1}\cap\Omega}_{2}}\left\{\abs{\nabla\phi(z)}\right\}\norm{u}_{\Omega_{1}\cap\Omega_{2}}^{2}\\
 &\leq& C(\norm{\overline{\partial}u}_{\Omega_{1}\cap\Omega_{2}}^{2}+\norm{\overline{\partial}^{\ast}u}_{\Omega_{1}\cap\Omega_{2}}^{2}). 
\end{eqnarray}
Thus, combining estimates at (\ref{eq53}) and (\ref{eq56}) we have 
\begin{eqnarray}
\norm{\phi u}_{\Omega_{1}\cap\Omega_{2}}^{2}&\leq \varepsilon(\norm{\overline{\partial}u}_{\Omega_{1}\cap\Omega_{2}}^{2}+\norm{\overline{\partial}^{\ast}u}_{\Omega_{1}\cap\Omega_{2}}^{2})+C_{\varepsilon}\norm{u}_{-1,\Omega_{1}\cap\Omega_{2}}^{2}
\end{eqnarray}
and
\begin{eqnarray}
\norm{(\phi-1) u}_{\Omega_{1}\cap\Omega_{2}}^{2}&\leq \varepsilon(\norm{\overline{\partial}u}_{\Omega_{1}\cap\Omega_{2}}^{2}+\norm{\overline{\partial}^{\ast}u}_{\Omega_{1}\cap\Omega_{2}}^{2})+C_{\varepsilon}\norm{u}_{-1,\Omega_{1}\cap\Omega_{2}}^{2}.
\end{eqnarray} 
Therefore, the existence of a compactness estimate on $\Omega_{1}\cap\Omega_{2}$ follows.  

\end{proof}

Another setting we consider is when $\D_1\subset\D_2$ and $\D_1$ and $\D_2$ share a piece of each other's boundary such that the boundaries of both domains separate smoothly from each other, see Figure \ref{Fig3}. 
We also assume that $b\D_1-b\D_2$ is strongly pseudoconvex and a compactness estimate for the $\overline{\partial}$-Neumann operator holds on $\D_{2}$. Then we want to know if there is a compactness estimate on $\D_1$. 

The following example demonstrates that such a setting exists. Although the domains in the example are smooth and convex, it is also possible to replace them by a biholomorphic map with non-convex domains with the same way of sharing boundaries.

\begin{figure}[ht!]
\centering
\includegraphics[width=120mm]{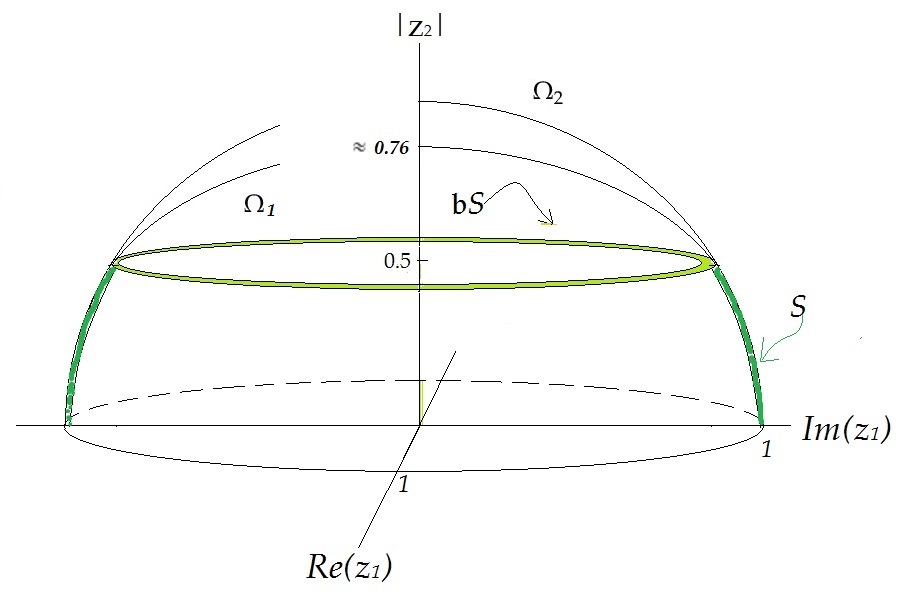}
\caption{$\D_1\subset\D_2$ and $\D_1$ and $\D_2$ share a piece of each other's boundary such that the boundaries of both domains separates smoothly from each other.}
\label{Fig3}
\end{figure}
\begin{example}
Let 
\[\lambda(t)=0\ \text{if}\ t\leq 0\ \&\ \lambda(t)=e^{-1/t}\ \text{if}\ t>0\]
$\lambda$ is a convex function on $(-\infty,1/2)$. Define

\[\D_1:=\left\{(z_1,z_2)\ :\ \rho_1(z_1,z_2)=\lambda\left(\frac{|z_1|^2+|z_2|^2}{3}\right)+\lambda\left(|z_2|^2-1/4\right)-e^{-3}<0\right\}\]
and 
\[\D_2:=\left\{(z_1,z_2)\ :\ \rho_2(z_1,z_2)=|z_1|^2+|z_2|^2-1<0\right\}.\]
Then $\D_1\subset\D_2$ and boundaries of $\D_1$ and $\D_2$ overlap as indicated in the paragraph before Proposition \ref{nontransversal intersection}. Let $S:=\lbrace\rho_1(z)=\rho_{2}(z)=0\rbrace\subset b\Omega_{2}$ be the part of the boundary with co-dimension 1. Then the boundary of $S$ is connected  (with codimension $2$), see the Figure \ref{Fig3}. In particular,
\[S=\left\{(z_1,z_2)\ :\ |z_1|^2+|z_2|^2=1\ \text{ and }\ |z_2|\leq 1/2 \right\}\ \text{ and }\]
\[bS=\left\{(z_1,z_2)\ :\ |z_1|=\frac{\sqrt{3}}{2}\ \text{ and }\ |z_2|= 1/2 \right\}.\]
\end{example}

Indeed, we get a compactness estimate on $\D_1$ by creating the setup in Proposition \ref{nontransversal intersection}. 
Consider the domain $\D_1$ as a result of a non-transversal intersection of $\D_2$ and another domain $\widetilde{\D_1}$ which shares part of the boundaries of $\D_1$ in a way that gives the setup in Proposition \ref{nontransversal intersection}.  

\begin{remark}
Going back to the first example, we can ask the question in the reverse direction.
Namely, let the $\dbar$-Neumann operator be compact on the non-transversal intersection of two domains also assume that non-intersecting parts of the boundary of both domains are strictly pseudoconvex. Can we conclude that the $\dbar$-Neumann operator is compact on each of the domains? Surprisingly the answer is unknown.
\end{remark}


\section{Exact regularity on the intersection of two domains}

An analog of the main problem associated with the Bergman projection operator can be formulated as follows. Again, let $\D_1$ and $\D_2$ be smooth bounded pseudoconvex domains, and assume that each has exactly regular Bergman projection operator. That is, the Bergman projection operators $\B_{\D_1}$ and $\B_{\D_2}$ map Sobolev space $W^k$ to Sobolev space $W^k$ for all $k\geq 0$. We are interested in if the Bergman projection on the intersection domain $\D_1\cap\D_2$ is also exactly regular. 

It is known that on a general Lipschitz domain the $\dbar$-Neumann operator (or even the Green operator for the Dirichlet problem) is not regular near the singular part of the domain, see \cite{BarretVassiliadou2003} and \cite{Shaw2005}. On the other hand, the Bergman projection on a product (which is Lipschitz) of smoothly bounded pseudoconvex domains (each having Condition-R) is exactly regular, see \cite[Corollary 1.3]{ChakrabartiShaw2011}. 

Locally, the Bergman projection is always regular \cite[Theorem $2^{\prime}$]{Barrett86}. However, on the transversal intersection of two balls, the Bergman projection is not exactly regular. In particular, near the non-generic points of the boundary, \cite{BarretVassiliadou2003} local Sobolev estimates fail.

In this section we present a similar example on the intersection of two polydiscs. It is known that the Bergman projection is regular on a polydisc. However, the Bergman projection on the intersection fails to be exactly regular. Our argument is elementary and is based on a straightforward biholomorphic equivalence.

Let $D_r(w_0)\subset \C$ denote the disc of radius $r>0$ centered at $w_0$. Define the following domains,
\begin{align*}
P_1^2&=\{(z_1,z_2)\in\C^2\ |\ z_1\in D_1(0),\ |z_2|<1\}\\
P_2^2&=\{(z_1,z_2)\in\C^2\ |\ z_1\in D_1(1),\ |z_2|<1\}\\
\mathcal{P}&:=P_1^2\cap P_2^2=\{(z_1,z_2)\in\C^2\ |\ z_1\in D_1(0)\cap D_1(1),\ |z_2|<1\}.
\end{align*}

\begin{theorem}\label{Sobolev Regularity}
The Bergman projection $\B_{\CP}$ of the intersection of two polydiscs is not exactly regular. In particular, there exists a smooth function on the closure of the intersection domain such that its Bergman projection is not a smooth function (or in $W^k$ for large enough $k$) on the closure of the domain.
\end{theorem}

\begin{proof}
We start the construction of the desired function in one dimension first. Since $\mathcal{P}$ is a product domain, we then lift the example up by a tensor argument. 

By the Riemann Mapping Theorem, there exits a conformal map $F(z)$ that maps $\dD=D_1(0)\cap D_1(1)$ onto the unit disc $\disc$. The conformal map $F:\dD\rightarrow\disc$ is explicitly given by 
\begin{align}\label{Conformal LanceToDisc}
 F(z)&=\frac{\left(\frac{2z-1+i\sqrt{3}}{2z-1-i\sqrt{3}}\right)^{\frac{3}{2}}+i}
{\left(\frac{2z-1+i\sqrt{3}}{2z-1-i\sqrt{3}}\right)^{\frac{3}{2}}-i}.
\end{align}

By \cite{Bell81} (see also \cite[Theorem 14.3.8]{KrantzBook}), the following transformation formula holds between the Bergman kernels of $\dD$ and $\disc$
\[K_{\dD}(z,w)= F^{\prime}(z)K_{\disc}( F(z), F(w))\overline{ F^{\prime}(w)}.\]

Since $\mathcal{P}$ is the product of domains $\dD$ and $\disc$ (in $z_1$ and $z_2$ complex planes respectively), we can calculate the Bergman kernel of $\mathcal{P}$ as a product. Namely,
\begin{align}\label{Bergam Kernel on Polydisc}
K_{\mathcal{P}}(z_1,z_2,w_1,w_2)&=K_{\dD}(z_1,w_1)\cdot K_{\disc}(z_2,w_2)\\
\nonumber&=\frac{1}{\pi^2}\frac{ F^{\prime}(z_1)\cdot\overline{ F^{\prime}(w_1)}}{\left(1- F(z_1)\overline{ F(w_1)}\right)^2}\cdot \frac{1}{(1-z_2\overline{w_2})^2}.
\end{align}

Next we take a smooth radial function $H(\zeta)$ on $\disc$ with compact support and $0\leq H(\zeta)\leq 1$. We define  
\begin{align}\label{test function}
\chi(z):=H\left(F(z)\right)\cdot F^{\prime}(z)\ \ \text{ on }\ \dD.
\end{align}
Note that $\chi(z)$ is a smooth function with compact support and we can calculate its projection onto the Bergman space.

\begin{lemma}\label{l3} 
Let $\chi(z)$ be defined as above. Then
\[\B_{\dD}\left(\chi\right)(z)=c\cdot F^{\prime}(z)\]
for some constant $c$.
\end{lemma} 
\begin{proof}
The Bergman projection of $\chi(z)$ on $\dD$ is given by
\begin{align*}
\B_{\dD}\left(\chi\right)(z)&=\int_{\dD} B_{\dD}(z,w)\cdot H\left(F(w)\right)\cdot F^{\prime}(w) dA(w).
\end{align*}

We switch the integration to $\disc$ by substituting $w=F^{-1}(\xi)$. Hence, $dA(w)=|(F^{-1}(\xi))^{\prime}|^2 dA(\xi)$ and

\begin{align*}
\B_{\dD}\left(\chi\right)(z)&=\int_{\disc} B_{\dD}(z,F^{-1}(\xi))\cdot H(\xi)\cdot F^{\prime}(F^{-1}(\xi))\cdot \overline{(F^{-1}(\xi))^{\prime}} \cdot (F^{-1}(\xi))^{\prime} dA(\xi)
\end{align*}

Also let  $z=F^{-1}(\zeta)$ such that $F(z)=\zeta$. Then

\begin{align*}
\B_{\dD}&\left(\chi\right)(F^{-1}(\zeta))=\\
\nonumber&=\frac{(F^{-1}(\zeta))^{\prime}}{(F^{-1}(\zeta))^{\prime}}\int_{\disc} B_{\dD}(F^{-1}(\zeta),F^{-1}(\xi))\cdot H(\xi)\cdot F^{\prime}(F^{-1}(\xi))\cdot \overline{(F^{-1}(\xi))^{\prime}} \cdot (F^{-1}(\xi))^{\prime} dA(\xi)
\end{align*}

By considering 
\[\frac{1}{F^{\prime}(F^{-1}(\xi))}=(F^{-1}(\xi))^{\prime}\]
the last expression is equal to

\begin{align*}
\B_{\dD}&\left(\chi\right)(F^{-1}(\zeta))
=\frac{(F^{-1}(\zeta))^{\prime}}{(F^{-1}(\zeta))^{\prime}}\int_{\disc} B_{\dD}(F^{-1}(\zeta),F^{-1}(\xi))\cdot H(\xi) \cdot \overline{(F^{-1}(\xi))^{\prime}} dA(\xi)\\
\nonumber&=\frac{1}{(F^{-1}(\zeta))^{\prime}}\cdot (F^{-1}(\zeta))^{\prime}\int_{\disc} B_{\dD}(F^{-1}(\zeta),F^{-1}(\xi))\cdot \overline{(F^{-1}(\xi))^{\prime}}\cdot H(\xi) dA(\xi)
\end{align*}

By using the transformation formula, we get
\begin{align*}
\B_{\dD}&\left(\chi\right)(F^{-1}(\zeta))=F^{\prime}(z)\int_{\disc} B_{\disc}(\zeta,\xi)\cdot H(\xi) dA(\xi).
\end{align*}
Since $H(\xi)$ is a radial function and compactly supported on $\disc$, its Bergman projection on $\disc$ is a nonzero constant. Thus, we get
\begin{align}
\B_{\dD}\left(\chi\right)(z)&=c\cdot F^{\prime}(z)
\end{align}

\end{proof}

Recall $\chi(z)=H(F(z))\cdot F^{\prime}(z)$ is smooth on the closure of $\mathcal{D}$. However, $F^{\prime}(z)$ has singularities at the corners of the domain, 
\begin{align*}
a=\frac{1}{2}-i\frac{\sqrt{3}}{2}\ \ \ \&\ \ \ 
b=\frac{1}{2}+i\frac{\sqrt{3}}{2}. 
\end{align*}
In particular, $\chi(z)\in W^{k}(\dD)$, for all $k\in \n$ but $F^{\prime}(z)\not\in W^{k}(\dD)$ for sufficiently large $k\in\n$ and $F^{\prime}(z)$ is not smooth on $\overline{\dD}$. 

It follows that the Bergman projection $\B_{\mathcal{P}}$ does not take $W^k(\mathcal{P})$ to $W^k(\mathcal{P})$ for some $k>0$, so $\B_{\mathcal{P}}$ is not exactly regular. Indeed, the intersection domain $\mathcal{P}$ is a product of $\mathcal{D}$ and $\disc$, we can redefine $\chi(z_1)$ as a function on $\mathcal{P}$ by considering it constant on $z_2$ direction. The Bergman projection on $\mathcal{P}$ of the redefined $\widetilde{\chi}(z_1,z_2)$ is the same as the Bergman projection of $\chi(z_1)$ on $\dD$, that is, $\B_{\mathcal{P}}\left(\widetilde{\chi}(z_1,z_2)\right)=\B_{\mathcal{D}}\left(\chi(z_1)\right)$. Therefore, again the singularities of $F'(z)$ break the regularity.

\end{proof}

\begin{remark}
It is known that the Bergman project is $L^p$- regular on a polydisc, see \cite{RudinDisc}. The same example of functions above also indicate that the $L^p$ regularity of the Bergman projection is not preserved under intersection. Indeed, due to compact support, $\widetilde{\chi}(z_1,z_2)$ is in $L^p(\mathcal{P})$ for all $p>1$. However, due to the singularities of $F'(z)$, the projection $B_{\mathcal{P}}\left(\widetilde{\chi}(z_1,z_2)\right)$ is not in $L^p(\mathcal{P})$ for sufficiently large $p$.

\end{remark}


\section{Hilbert-Schmidt property of Hankel operators on the intersection of two domains}

A linear bounded operator $T$ on a Hilbert space $H$ is called a \textit{Hilbert-Schmidt operator} if there is an orthonormal basis $\{\xi_{j}\}$ for $H$ such that the sum $\sum\limits_{j=1}^{\infty}\norm{T(\xi_{j})}^2$ is finite. 
It is known that any Hilbert-Schmidt operator is compact and they are dense in the space of compact operators, see \cite[Section X.]{Retherford93}. 

Another analog of the main intersection problem associated to Hilbert-Schmidt property of Hankel operators can be formulated as follows. Let a Hankel operator $H_{\phi}$ be Hilbert-Schmidt on the Bergman spaces of two domains $\D$ and $\D^{\prime}$,  is it also Hilbert-Schmidt on the Bergman space of the intersection domain $\D\cap\D^{\prime}$? Below we answer this question in negative by showing an explicit example. The domains and computations below follow the ideas previously presented in \cite{Wiegerinck84, CelikZeytuncuACTA,CelikZeytuncuRMJM}.

Set
\begin{align*}
X&=\left\{(z_1,z_2)\in\C^2\ : |z_1|>4,\ \ |z_2|<\frac{1}{2|z_1|\log_4|z_1|}\right\}\\
Y&=\left\{(z_1,z_2)\in\C^2\ : |z_2|>4,\ \ \left||z_1|-\frac{1}{|z_2|}\right|<\frac{1}{|z_2|^3}\right\}\\
Z&=\left\{(z_1,z_2)\in\C^2\ : |z_1|\leq 4,\ \ |z_2|\leq 4\right\}
\end{align*}
and define 
\[\D=X\cup Y\cup Z.\]

Also set
\begin{align*}
X^{\prime}&=\left\{(z_1,z_2)\in\C^2\ : |z_1|>4,\ \ \left||z_2|-\frac{1}{|z_1|}\right|<\frac{1}{|z_1|^3}\right\}\\
Y^{\prime}&=\left\{(z_1,z_2)\in\C^2\ : |z_2|>4,\ \ |z_1|<\frac{1}{2|z_2|\log_4|z_2|}\right\}\\
Z&=\left\{(z_1,z_2)\in\C^2\ : |z_1|\leq 4,\ \ |z_2|\leq 4\right\}
\end{align*}
and define
\[\D^{\prime}=X^{\prime}\cup Y^{\prime}\cup Z.\]

Note that both $\D$ and $\D^{\prime}$ are unbounded Reinhardt domains with finite volumes, see Figure \ref{Fig4}. 

\begin{figure}[ht!]
\centering
\includegraphics[width=150mm]{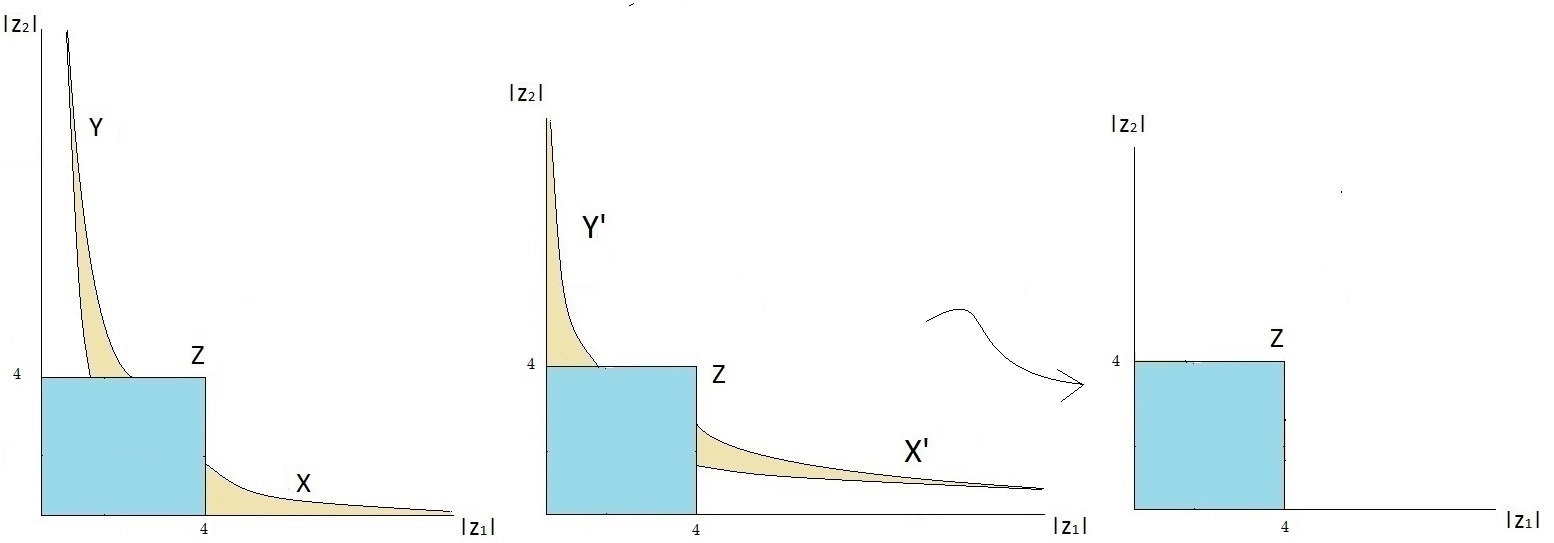}
\caption{$\D$ and $\D^{\prime}$ are unbounded Reinhardt domains with finite volumes and $\D\cap\D^{\prime}$ is a polydisc.}
\label{Fig4}
\end{figure}


It is evident that $X\cap X^{\prime}=\emptyset$. Indeed, the maximum radius of $X$ in $z_2$ direction is $\frac{1}{2|z_1|\log_4|z_1|}$ and the minimum radius of $X^{\prime}$ on $z_2$ direction is $\frac{1}{|z_1|}-\frac{1}{|z_1|^3}$. If $|z_1|>4$, since $$\frac{1}{2|z_1|\log_4|z_1|}<\frac{1}{|z_1|}-\frac{1}{|z_1|^3}$$ we conclude that $X\cap X^{\prime}=\emptyset$.  A similar argument shows $Y\cap Y^{\prime}=\emptyset$.
Thus, $\D\cap\D^{\prime}=Z$. 

Due to the construction of two domains, the Bergman spaces $A^2(\D)$ and $A^2(\D^{\prime})$ have special properties. In particular, both spaces are spanned by the monomials of the form $(z_1 z_2)^j$. A straightforward computation indicates (see also below), $$\int_{\D}|z_1^jz_2^k|^2dV(z) \text{ and }\int_{\D^{\prime}}|z_1^jz_2^k|^2dV(z)$$
are finite if and only if $j=k$. Using the radial symmetry of domains, any holomorphic function can be expanded into Taylor series and we obtain the orthogonal basis $\left\{(z_1 z_2)^j\right\}_{j=1}^{\infty}$ for $A^2(\D)$ and $A^2(\D^{\prime})$. 

On two Bergman spaces we consider the Hankel operator $H_{\zb_1\zb_2}$. Although the symbol $\zb_1\zb_2$ is not bounded on $\D$ or $\D'$, the operator is bounded on both Bergman spaces. This follows from comparing the norms of the monomials in the Bergman space, it becomes clear in the proof of Theorem \ref{HS}. We further prove the following.

\begin{theorem}\label{HS}
$H_{\zb_1\zb_2}$ is Hilbert-Schmidt on $A^2(\D)$ and $A^2(\D^{\prime})$. However, it is not on $A^2(\D\cap\D^{\prime})$.
\end{theorem}

\begin{proof}
The set $\left\{\frac{z_1^k z_2^k}{c_{(k,k)}}\right\}_{k\in \n}$  is an orthonormal basis for $A^{2}(\D)$ and $A^2(\D^{\prime})$ where $$c_{(k,k)}^2=\int_{\D}|z_1z_2|^{2k}dV(z_1,z_2)=\int_{\D^{\prime}}|z_1z_2|^{2k}dV(z_1,z_2).$$ In order to prove that $H_{\zb_1\zb_2}$ is a Hilbert-Schmidt Hankel operator on both spaces, we look at the sum 
\[
\sum\limits_{k=1}^{\infty}\left\|H_{\zb_1\zb_2}\left(\frac{z_1^k z_2^k}{c_{(k,k)}}\right)\right\|^2
\]
where the norms are identical on $\D$ and $\D^{\prime}$. Indeed,

\begin{align*}
\left\|H_{\zb_1\zb_2}\left(\frac{z_1^k z_2^k}{c_{(k,k)}}\right)\right\|^2&=\frac{1}{c_{(k,k)}^2}\left\langle H_{\zb_1\zb_2}\left(z_1^k z_2^k\right),H_{\zb_1\zb_2}\left(z_1^k z_2^k\right)\right\rangle\\
\nonumber&=\frac{1}{c_{(k,k)}^2}\left\langle \zb_1\zb_2 z_1^k z_2^k-P(\zb_1\zb_2 z_1^k z_2^k), \zb_1\zb_2 z_1^k z_2^k-P(\zb_1\zb_2 z_1^k z_2^k)\right\rangle\\
\nonumber&=\frac{1}{c_{(k,k)}^2}\left(c_{(k+1,k+1)}^2-\frac{c_{(k,k)}^2 c_{(k,k)}^2}{c_{(k-1,k-1)}^2}-\frac{c_{(k,k)}^2 c_{(k,k)}^2}{c_{(k-1,k-1)}^2}+\frac{c_{(k,k)}^2 c_{(k,k)}^2c_{(k-1,k-1)}^2}{c_{(k-1,k-1)}^2 c_{(k-1,k-1)}^2}\right)\\
\nonumber&=\frac{c_{(k+1,k+1)}^2}{c_{(k,k)}^2}-\frac{c_{(k,k)}^2}{c_{(k-1,k-1)}^2}.
\end{align*}

Therefore we need to estimate the sum
\begin{align}\label{norm of Hankel - last step}
\sum_{k=1}^{\infty}\left(\frac{c_{(k+1,k+1)}^2}{c_{(k,k)}^2}-\frac{c_{(k,k)}^2}{c_{(k-1,k-1)}^2}\right).
\end{align}

We look at terms $c_{(k,k)}^2$, by computing three integrals on separate pieces of the domains. First,
\begin{align*}
\int\limits_{X}|z_1z_2|^{2k}dV(z_1,z_2)
&=\frac{4\pi^{2}}{2k+2}\int\limits_{4}^{\infty} \frac{1}{2r_1(\log_4(r_1))^{2(2k+2)}} dr_1
=\frac{4\pi^{2}}{2k+2}\left(-\frac{\log_4^{-2k-1}(r_1)}{2k+1}\right)\Bigg\vert_{4}^{\infty}\\
&= \frac{4\pi^2}{2(2k+2)(2k+1)}.
\end{align*}
Next,
\begin{align*}
\int\limits_{Y}|z_1z_2|^{2k}dV(z_1,z_2) &=4\pi^2\int\limits_{4}^{\infty} r_2^{2k+1}\int\limits_{\frac{1}{r_2}-\frac{1}{r_2^3}}^{\frac{1}{r_2}+\frac{1}{r_2^3}} r_1^{2k+1}dr_1dr_2\\
&=4\pi^2\frac{1}{2k+2}\int\limits_{4}^{\infty} r_2^{2k+1}\left[\left(\frac{r_2^2+1}{r_2^3}\right)^{2k+2}-\left(\frac{r_2^2-1}{r_2^3}\right)^{2k+2}\right] dr_2\\
&=4\pi^2\frac{1}{2k+2}\int\limits_{4}^{\infty} r_2^{2k+1}\left[\frac{2(2k+2)r_2^{2(2k+1)}+\text{\emph{(lower order terms)}}}{r_2^{6k+6}}\right] dr_2\\
&=8\pi^2\int\limits_{4}^{\infty} \frac{1}{r_2^{3}} dr_2 
+ 4\pi^2\frac{1}{2k+2}\int\limits_{4}^{\infty} r_2^{2k+1}\frac{\text{\emph{(lower order terms)}}}{r_2^{6k+6}} dr_2\\
&=\frac{\pi^2}{4}+\text{\emph{(lower order terms)}}=:\beta_k\\
\end{align*}
where $\beta_k$ is bounded from below and above. Finally,
\begin{align*}
\int\limits_{Z}|z_1z_2|^{2k}dV(z_1,z_2)=(2\pi)^2 \left(\int_0^4 r_1^{2k+1}dr_1 \right)^2= \frac{4\pi^2 \cdot 2^{8k+6}}{(k+1)^2}.
\end{align*}
When we add three pieces together,
\begin{align}
\nonumber c_{(k,k)}^2&=\int\limits_{X}|z_1z_2|^{2k}dV(z_1,z_2)+\int\limits_{Y}|z_1z_2|^{2k}dV(z_1,z_2)+\int\limits_{Z}|z_1z_2|^{2k}dV(z_1,z_2)\\
\nonumber &=4\pi^2\left[\frac{1}{2(2k+2)(2k+1)}+\beta_k+\frac{2^{8k+6}}{(k+1)^2}\right]\\
\nonumber&=4\pi^2\left[\frac{\gamma_k}{(k+1)^2}+\beta_k+\frac{2^{8k+6}}{(k+1)^2}\right]\ \ \text{where} \ \gamma_k\approx 1\ \&\ \beta_k\approx 1.
\end{align}

We plug everything back in
\begin{align*}
\frac{c_{(k+1,k+1)}^2}{c_{(k,k)}^2}-\frac{c_{(k,k)}^2}{c_{(k-1,k-1)}^2}=\frac{c_{(k+1,k+1)}^2 c_{(k-1,k-1)}^2-c_{(k,k)}^2 c_{(k,k)}^2}{c_{(k,k)}^2 c_{(k-1,k-1)}^2}.
\end{align*}

A straightforward computation gives that  

\begin{align*}c_{(k+1,k+1)}^2 c_{(k-1,k-1)}^2-c_{(k,k)}^2 c_{(k,k)}^2&=\frac{2^{16k+12}\left((k+1)^4-(k+2)^2k^2\right)}{(k+2)^2 k^2 (k+1)^4}+\frac{\text{\emph{(lower order terms)}}}{(k+2)^2 k^2 (k+1)^4}\\
&=\frac{2^{16k+12}\left(2k^2+4k+1\right)}{(k+2)^2 k^2 (k+1)^4}+\frac{\text{\emph{(lower order terms)}}}{(k+2)^2 k^2 (k+1)^4}
\end{align*} 

and 
$$c_{(k,k)}^2 c_{(k-1,k-1)}^2=\frac{2^{16k+4}}{(k+1)^2 k^2}+\frac{\text{\emph{(lower order terms)}}}{(k+1)^2 k^2}.$$
Therefore, 
\begin{align*}
\frac{c_{(k+1,k+1)}^2}{c_{(k,k)}^2}-\frac{c_{(k,k)}^2}{c_{(k-1,k-1)}^2}=\frac{1}{(k+2)^2(k+1)^2}\cdot \frac{2^{16k+12}\left[2k^2+4k+1\right]+\cdots}{2^{16k+4}+\cdots}\approx \frac{1}{k^2}.
\end{align*}

Hence, the sum in the equation \eqref{norm of Hankel - last step} is finite and so the Hankel operator $H_{\zb_1\zb_2}$ is Hilbert-Schmidt on $A^2(\D)$ and $A^2(\D^{\prime})$.  For more on \eqref{norm of Hankel - last step} and detailed computations see \cite{CelikZeytuncu2013}.

On the other hand, $H_{\zb_1\zb_2}$ is not a Hilbert-Schmidt on the Bergman space of the intersection domain $\D\cap\D^{\prime}=Z$. In fact, none of the Hankel operators with anti-holomorphic symbols is Hilbert-Schmidt on the Bergman space of a bounded Reinhardt domain, see \cite{CelikZeytuncuACTA,Trieu2014}. 
\end{proof}



\section{Further Directions}

The compactness of the $\dbar$-Neumann operator is one of main point of further investigation in this context. The similar question for Hankel operators is also of interest. The compactness of Hankel operators is also a local property, see \cite{Sahut}. Therefore, one can investigate if the compactness of a Hankel operator is preserved on an intersection domain.

We used unbounded domains in the fourth section to have Hilbert-Schmidt Hankel operators. It is not known if other Hilbert-Schmidt Hankel operators can be constructed on bounded domains. The answer is negative if the symbol is anti-holomorphic and domains are Reinhardt, see \cite{CelikZeytuncu2013,CelikZeytuncuACTA,Trieu2014}. However, the general case is unknown.

\section*{Acknowledgements}
The authors would like to thank the anonymous referee and the editor for constructive comments. The second author would like to thank the organizers of the Analysis and Geometry in Several Complex Variables Conference at Texas A\&M University at Qatar, Shiferaw Berhanu, Nordine Mir and Emil J. Straube for the kind invitation. The first author would like to thank organizers of the workshop on several complex variables and CR geometry at Erwin Schr$\ddot{\text{o}}$dinger International Institute for Mathematical Physics in Vienna, Austria in $2015$ for the thought provoking talks and productive environment.


\end{document}